\documentclass[12pt]{amsart}
\usepackage{geometry, amsmath, amsthm,amsfonts,amssymb,stmaryrd, mathrsfs, graphicx}
\usepackage{fancyhdr}
\usepackage{hyperref}
\usepackage{fullpage}
\usepackage{tikz-cd}
\usepackage[all]{xy}
\usepackage{dynkin-diagrams}
\usetikzlibrary{patterns}
\usetikzlibrary{decorations.pathreplacing,calligraphy,decorations.pathmorphing}
\usetikzlibrary{intersections}

\usetikzlibrary{matrix,decorations.pathreplacing,calc}

\newtheorem{theorem}[subsubsection]{Theorem}
\newtheorem{lemma}[subsubsection]{Lemma}
\newtheorem{corollary}[subsubsection]{Corollary}
\newtheorem{conjecture}[subsubsection]{Conjecture}

\newtheorem{proposition}[subsubsection]{Proposition}
\theoremstyle{definition}

\newtheorem{definition}[subsubsection]{Definition}

\newtheorem{remark}[subsubsection]{Remark}

\newtheorem{notation}[subsubsection]{Notation}

\newtheorem{fact}[subsubsection]{Fact}

\numberwithin{equation}{subsubsection}

\def\CC{\mathbb{C}}

\def\ZZ{\mathbb{Z}}

\title{On the total positivity of contingency metamatrices}
\author{Zhentao Wang, Jiawen Xie and XUhang Zhang}

\begin{document}
\maketitle

\begin{abstract}
M. Kapranov and V. Schechtman introduced the contingency metamatrix for a finite Coxeter group and conjectured that the contingency metamatrix is totally positive. For the Coxeter groups of type $A$, this conjecture has been proved by P. Etingof. In this article, we prove this conjecture for the Coxeter groups of type $B$ and exceptional types.
\end{abstract}

\section{Introduction}

In \cite{ks}, M. Kapranov and V. Schechtman constructed a refinement of the natural stratification of $\CC^n/S_n$ according to the patterns of real parts and imaginary parts. 
The numbers of strata of different bi-degrees of this stratification form the so-called contingency metamatrix
\begin{equation*}
    M=(M_{pq})_{0\leq p,q\leq n-1},
\end{equation*} 
where the number $M_{pq}$ of strata of bi-degree $(p,q)$ is given by 
\begin{equation*}
    M_{pq}=\sum_{|\alpha|=p,|\beta|=q}|S_{\alpha}\backslash S_n/S_{\beta}|.
\end{equation*}
Here, $S_n=\langle s_1,\dots,s_{n-1}|s_i^2=(s_is_{i+1})^3=(s_is_j)^2=1, \forall |i-j|>1\rangle$ is the symmetric group and $S_\alpha$ and $S_\beta$ are parabolic subgroups of $S_n$, where $\alpha,\beta\subset \{s_1,\dots,s_{n-1}\}$ are of cardinalities $p$ and $q$, respectively.

\par A striking result, due to P. Etingof \cite[Appendix, Corollary A.13]{ks}, shows that the contingency metamatrix is totally positive. M. Kapranov and V. Schechtman conjectured that total positivity holds for any real reflection group $W$, not only for type $A$ (i.e., $W=S_n$):
\begin{conjecture}[M. Kapranov, V.~Schechtman]\label{conjecture}
   Let $(W,S)$ be a finite Coxeter group of rank $|S|=n$. Then the metamatrix $M(W)=(M_{pq})_{0\leq p,q\leq n}$, $$M_{pq}=\sum_{I,J\subset S,|I|=p,|J|=q}|W_{I}\backslash W/W_{J}|$$ is totally positive, where $W_I$ and $W_J$ are parabolic subgroups of $W$ of type $I$ and $J$, respectively.
\end{conjecture}

Via private communication with M.~Finkelberg, we are informed that I. Ukraintsev proved the conjecture for the Coxeter groups whose ranks are no more than $3$.

In this article, we proved the conjecture for type $B$ and exceptional types. Namely,
\begin{theorem}\label{main theorem}
    Let $(W,S)$ be of type $B_n$ or of exceptional type, then the metamatrix $M(W)$ is totally positive.
\end{theorem}

Roughly speaking, we follow the method of P.~Etingof for type $A$ to give a proof for type $B$. We first show that there is a bijection between the set of double parabolic cosets and the set of signed contingency matrices. Then we compute the number of generalized signed contingency matrices, which helps us to compute the number of signed contingency matrices. Finally, we show the total positivity of the metamatrices via Whitney's theorem\cite{c}. For exceptional types, we verify the conjecture using computational methods.




\section{Proof For Type  \texorpdfstring{$B$}{B}}\label{B}

Let $(W,S)$ be the Coxeter group of type $B_n$, where $S:=\{s_1,s_2,\dots, s_n\}$ is the set of simple reflections. We have the following relations:
\begin{equation*}
 \begin{aligned}
    (s_{i}s_{i+1})^{3}=&1,\ \mathrm{if} \ 0\leq i<n-2, \\
    (s_{n-1}s_{n})^{4}=&1,\\
    (s_{i}s_{j})^{2}=&1,\ \mathrm{if}\ |i-j|>1.\\
 \end{aligned}   
\end{equation*}

\subsection{Double cosets by parabolic subgroups}

\par Let $E_{ij}$ be the matrix with the $(i,j)$-entry being $1$ and the other entries being $0$. 
\par We have a useful combinatorial description of $W$:

\begin{lemma}\cite[Section 8.1]{bb}
    There is a faithful representation of $W$ on $\mathbb{R}^n$
    \begin{equation*}
        \phi\colon W\hookrightarrow \mathrm{GL}_{n}(\mathbb{R}),
    \end{equation*}
    defined by
    \begin{equation*}
        \begin{aligned}
            \phi(s_{i})=& I_n-E_{ii}-E_{i+1,i+1}+E_{i,i+1}+E_{i+1,i} ,\ 1\le i<n,\\
            \phi(s_{n})=& I_n-2E_{nn}.
        \end{aligned}
    \end{equation*}
\end{lemma}

In the following, we identify $W$ with $\phi(W)$. 

\begin{remark}
    By this lemma, we can give a concrete description (see Proposition \ref{3.1.8}) of the parabolic double coset $W_{I}\backslash W/W_{J}$, where $I,J\subset S$.
    Although the proof is a little technical, the underlying idea is very simple and direct. Before going into the rigorous description, we first give an informal account. For $0< i< n$, the left (resp. right) action of $s_{i}$ on $W\subset \mathrm{GL}_{n}(\mathbb{R})$ changes the $i$-th row (resp. column) and $(i+1)$-st row (resp. column) of the matrices. So we can identify these two rows (resp. columns) if we take the quotient by $s_{i}$ from the left (resp. right). Similarly, the left (resp. right) action of $s_n$ changes the sign of elements in the last row (resp. column). Hence, once we take the quotient by $s_n$ from the left (resp. right), we can ignore the sign of elements in the last row (resp. column). Using this observation, we can obtain a useful description for the set $W_{I}\backslash W/W_{J}$.    
\end{remark}

\begin{notation}
    The set of signed numbers is
    \begin{equation*}
        \mathbb{S}:=\{(a^+,a^-):a^+,a^- \in \ZZ_{\ge0}\}.
    \end{equation*}
    Let $a=(a^+,a^-)\in \mathbb{S}$, we denote by $|a|:=a^+ + a^-$ the absolute value of $a$.
\end{notation}

\begin{definition}[margin conditions] 
    A margin condition of $n$ of length $p$ is a pair $\alpha = (\tilde{\alpha},\lambda_\alpha)$, where $(\tilde{\alpha},\lambda_\alpha)$ is
    \begin{itemize}
        \item either $ \tilde{\alpha} $ is an ordered partition of $n$ of length $p$, $\lambda_\alpha = 0$; 
        \item or $ \tilde{\alpha} $ is an ordered partition of $n$ of length $p+1$, $\lambda_\alpha = 1$.
    \end{itemize}  
    We denote the length of $\alpha$ by $|\alpha|$.   
\end{definition}

\begin{proposition}
    There is a natural bijection 
    \begin{equation*}
        \{\text{subset $I$ of $S$, $|I| = p$}\} \xrightarrow{\sim} \{\text{margin condition $\alpha$ of $n$ of length $n-p$ }\}.
    \end{equation*}
\end{proposition}

\begin{proof}
    For $I\subset S$, let $\tilde{\alpha}$ be the unique partition of $n$, such that $i$ and $i+1$ lie in the same part if and only if $s_i\in I$, for $1\leq i\leq n-1$; let $\lambda_\alpha= |I\cap {s_n}|$. Then the bijection sends $I$ to $(\tilde{\alpha},\lambda_\alpha)$. 
    \par On the other hand, for $(\tilde{\alpha},\lambda_\alpha)$ a margin condition of $n$, then $s_i$ lies in the corresponding subset $I$ if and only if $i$ and $i+1$ lie in the same part of $\tilde{\alpha}$, for $1\leq i\leq n-1$; and $s_n\in I$ if and only if $\lambda_\alpha=1$.

\end{proof}

\begin{definition}[signed contingency matrix]\label{signed}
    Let $\alpha = (\tilde{\alpha},\lambda_\alpha),\beta = (\tilde{\beta},\lambda_\beta)$ be two margin conditions of $n$ of length $p,q$, respectively. The set of signed contingency matrices with margins $(\alpha,\beta)$ is
    \begin{itemize}
        \item[(1)] $\lambda_\alpha=0$, $\lambda_\beta=0$ : 
            \begin{equation*}
                \mathrm{SCM}_n(\alpha,\beta):=
                \{A\in \mathrm{Mat}_{p,q}(\mathbb{S}): \sum_{j} |a_{ij}|=\alpha_{i}, \sum_{i}|a_{ij}|=\beta_{j}\}.
            \end{equation*}
        \item[(2)] $\lambda_\alpha=1$, $\lambda_\beta=0$ : 
            \begin{equation*}
                \mathrm{SCM}_n(\alpha,\beta):=
                \{A\in \mathrm{Mat}_{p+1,q}(\mathbb{S}): \sum_{j} |a_{ij}|=\alpha_{i},\sum_{i}|a_{ij}|=\beta_{j},a^-_{p+1,j}=0\}.
            \end{equation*}
        \item[(3)] $\lambda_\alpha=0$, $\lambda_\beta=1$ : 
            \begin{equation*}
                \mathrm{SCM}_n(\alpha,\beta):=
                \{A\in \mathrm{Mat}_{p,q+1}(\mathbb{S}): \sum_{j} |a_{ij}|=\alpha_{i}, \sum_{i}|a_{ij}|=\beta_{j},a^-_{i,q+1}=0\}.
            \end{equation*}
        \item[(4)] $\lambda_\alpha=1$, $\lambda_\beta=1$ : 
            \begin{equation*}
                \mathrm{SCM}_n(\alpha,\beta):=
                \{A\in \mathrm{Mat}_{p+1,q+1}(\mathbb{S}): \sum_{j} |a_{ij}|=\alpha_{i},\sum_{i}|a_{ij}|=\beta_{j},a^-_{p+1,j}= a^-_{i,q+1}=0\}.
            \end{equation*}
    \end{itemize}
    
    For given lengths $p$ and $q$, let us denote
    \begin{equation*}
        \mathrm{SCM}_{n}(p,q):=\bigcup_{|\alpha|=p,|\beta|=q} \mathrm{SCM}_{n}(\alpha,\beta).
    \end{equation*}
\end{definition}

\begin{definition}
    \begin{enumerate}
        \item Fix a margin condition $\alpha$, a colored signed ordered partition of $n$ of type $\alpha$ is an ordered set $v=(v_1,\cdots,v_{|\tilde{\alpha}|})$ of sets, where $v_i$ is a subset of $\{1,\dots,n\}\times\{\pm\}$ of cardinality $\tilde{\alpha}_i$, such that under the projection 
        \begin{equation*}
            p:\{1,\dots,n\}\times\{\pm\}\rightarrow \{1,\dots,n\},\quad (a,sign)\mapsto a,
        \end{equation*}
        $\{p(v_i)\}_{1\leq i\leq |\tilde{\alpha}|}$ forms a disjoint decomposition of $\{1,\dots,n\}$.
        \item Fix margin conditions $\alpha$ and $\beta$ , a colored double signed contingency matrix of $n$ of margin $(\alpha,\beta)$ is a matrix $A=(A_{ij})_{1\leq i\leq |\tilde{\alpha}|,1\leq j\leq |\tilde{\beta}|}$, where $A_{ij}$ is a subset of $\{1,\dots,n\}\times\{\pm\}\times \{\pm\}$, such that under the projections 
        \begin{eqnarray*} 
            p_1:\{1,\dots,n\}\times\{\pm\}\times \{\pm\}\rightarrow \{1,\dots,n\}\times\{\pm\},\quad (a,sign_1,sign_2)\mapsto (a,sign_1),\\
            p_2:\{1,\dots,n\}\times\{\pm\}\times \{\pm\}\rightarrow \{1,\dots,n\}\times\{\pm\},\quad (a,sign_1,sign_2)\mapsto (a,sign_2),\\
        \end{eqnarray*}
        $(\cup_jp_1(A_{1j}),\dots,\cup_jp_1(A_{|\tilde{\alpha}|j}))$ forms a colored signed ordered partition of $n$ of type $\alpha$ and $(\cup_ip_2(A_{i1}),\dots,\cup_ip_2(A_{i|\tilde{\beta}|})$ forms a colored signed ordered partition of $n$ of type $\beta$.
    \end{enumerate}
\end{definition}

\begin{proposition}\label{3.1.8}
    Fix two subsets $I$ and $J$ of $S$, there is a natural bijection 
    \begin{equation*}
        W_I\backslash W/ W_J \xlongrightarrow{\sim} \mathrm{SCM}_{n}(\alpha_I,\beta_J).
    \end{equation*}
\end{proposition}

\begin{proof}
    Let $W$ act diagonally on $W/W_I\times W/W_J$ by $w'\cdot(w_1W_I,w_2W_J)=(ww_1W_I,ww_2W_J)$.\\
    There is a bijection of sets 
    \begin{equation*}
        W\backslash (W/W_I\times W/W_J)\longrightarrow W_I\backslash W / W_J \quad W(w_1W_I,w_2W_J) \mapsto W_Iw_1^{-1}w_2W_J.
    \end{equation*}
    By the matrix interpretation of $W$, $W/W_I$ can be interpreted as the set of colored signed ordered partitions of $n$ of type $\alpha_I$.
    Under this interpretation, $W/W_I\times W/W_J$ has a canonical bijection to the set $C_n(\alpha_I,\beta_J)$ of colored double signed contingency matrix of $n$ of margin $(\alpha_I,\beta_J)$:
    \begin{equation*}
        C_n(\alpha_I,\beta_J) \longrightarrow W/W_I\times W/W_J,\quad A\mapsto ((\cup_jp_1(A_{1j}),\dots),(\cup_ip_2(A_{i1}),\dots)).
    \end{equation*}
    Therefore we have 
    \begin{equation*}
        \mathrm{SCM}_n(\alpha_I,\beta_J) = W \backslash C_n(\alpha_I,\beta_J) = W \backslash  (W/W_I\times W/W_J) = W_I\backslash W/ W_J.
    \end{equation*}
    
\end{proof}

\begin{remark}
    Via private communication with M.~Finkelberg, we are informed that I.~Ukraintsev also independently found the above description of the double cosets $W_I\backslash W/W_J$.
\end{remark}

By Proposition \ref{3.1.8}, we have another description of the entry of contingency metamatrix
\begin{equation*}
    M_{pq}= |\mathrm{SCM}_{n}(n-p,n-q)|.
\end{equation*}

\subsection{Calculate the metamatrix}


    Signed contingency matrices can also be characterized as matrices with every row and column being nonzero with some extra conditions. Removing the restriction that every row and column cannot be zero, we have the following definition.
\begin{definition}
    $\mathrm{(generalized\  signed\  contingency\  matrices)}$
    The set of generalized signed contingency matrices is defined as
    \begin{equation*}
        \mathrm{GSCM}_n(p,q):= \bigcup_{0\leq i,j\leq 1} \mathrm{GSCM}_n^{i,j}(p,q),
    \end{equation*}
    where 
    \begin{align*}
        \mathrm{GSCM}_n^{0,0}(p,q)&:=\{A\in \mathrm{Mat}_{p,q}(\mathbb{S}):\sum_{i,j}|a_{ij}|=n\};\\
        \mathrm{GSCM}_n^{1,0}(p,q)&:=\{A\in \mathrm{Mat}_{p+1,q}(\mathbb{S}):\sum_{i,j}|a_{ij}|=n,\; a_{p+1,j}^-=0,\; \sum_j a_{p+1,j}^+>0\};\\
        \mathrm{GSCM}_n^{0,1}(p,q)&:=\{A\in \mathrm{Mat}_{p,q+1}(\mathbb{S}):\sum_{i,j}|a_{ij}|=n,\; a_{i,p+1}^-=0,\; \sum_i a_{i,q+1}^+>0\};\\
        \mathrm{GSCM}_n^{1,1}(p,q)&:=\\\{A\in \mathrm{Mat}_{p+1,q+1}(\mathbb{S}):&\sum_{i,j}|a_{ij}|=n,\; a_{p+1,j}^-=a_{j,q+1}^-=0,\; \sum_j a_{p+1,j}^+>0,\;\sum_i a_{i,q+1}^+>0\}.
    \end{align*}
\end{definition}

There are several facts via combinatorial methods.
\begin{fact}
\begin{enumerate}
    \item The cardinalities of generalized signed contingency matrices $\mathrm{GSCM}_n^{\lambda,\mu}(p,q)$ ($\lambda,\mu\in \{0,1\}$) are 
        \begin{align*}
            |\mathrm{GSCM}_n^{0,0}(p,q)|=&\sum\limits_{a=0}^n\binom{a+pq-1}{a}\binom{n-a+pq-1}{n-a},\\
            |\mathrm{GSCM}_n^{1,0}(p,q)|=&\\ \sum\limits_{a=0}^n\binom{a+(p+1)q-1}{a}&\binom{n-a+pq-1}{n-a}-\sum\limits_{a=0}^n\binom{a+pq-1}{a}\binom{n-a+pq-1}{n-a},\\
            |\mathrm{GSCM}_n^{0,1}(p,q)|=&\\ \sum\limits_{a=0}^n\binom{a+(q+1)p-1}{a}&\binom{n-a+pq-1}{n-a}-\sum\limits_{a=0}^n\binom{a+pq-1}{a}\binom{n-a+pq-1}{n-a},\\
            |\mathrm{GSCM}_n^{1,1}(p,q)|=&\sum\limits_{a=0}^n\binom{a+(p+1)(q+1)-1}{a}\binom{n-a+pq-1}{n-a}\\
            &-\sum\limits_{a=0}^n\binom{a+p(q+1)-1}{a}\binom{n-a+pq-1}{n-a}\\
            &-\sum\limits_{a=0}^n\binom{a+(p+1)q-1}{a}\binom{n-a+pq-1}{n-a}\\
            &+\sum\limits_{a=0}^n\binom{a+pq-1}{a}\binom{n-a+pq-1}{n-a}.
        \end{align*}
        \item \begin{equation*}
            |\mathrm{GSCM}_n(p,q)|=\sum\limits_{a=0}^n\binom{a+(p+1)(q+1)-1}{a}\binom{n-a+pq-1}{n-a}.
        \end{equation*}
        \item The relation between $|\mathrm{SCM}_n(p,q)|$ and $|\mathrm{GSCM}_n(p,q)|$ is characterized by the following formula:
            \begin{equation*}
                |\mathrm{GSCM}_n^{\lambda,\mu}(p,q)|=\sum\limits_{0\le i\le p,0\le j\le q}\binom{p}{i}\binom{q}{j}|\mathrm{SCM}^{\lambda,\mu}_{n}(p,q)|,
            \end{equation*}
        where $0\leq \lambda,\mu \leq 1$. 
        \item 
            \begin{equation*}
                |\mathrm{GSCM}_n(p,q)|=\sum\limits_{0\le i\le p,0\le j\le q}\binom{p}{i}\binom{q}{j}|\mathrm{SCM}_{n}(p,q)|,
            \end{equation*}
\end{enumerate}
\end{fact}

\subsection{Total positivity}

The key observation to prove the total positivity is the following proposition.

\begin{proposition}
    \begin{itemize}
        \item[(1)] The matrix $L=(L_{pq})_{0\leq p,q\leq n}:=(|\mathrm{GSCM}_n(p,q)|)_{0\leq p,q\leq n }$ can be decomposed as:
        \begin{equation*}
            L=V\cdot D\cdot{}^tV,
        \end{equation*}
         where 
         \begin{equation*}
            V=((i+\frac{1}{2})^{j})_{0\le i,j\le n}
         \end{equation*}
         is a Vandermonde matrix and 
         \begin{equation*}
            D=diag(d(n,0),d(n,1),\dots,d(n,n)) 
         \end{equation*}
         is a diagonal matrix with every $d(n,i)$ being positive.
        \item[(2)] $Q:=P^{-1}\cdot V$ is upper triangular, where $P=(\binom{i}{j})_{0\le i,j\le n}$.
    \end{itemize}
\end{proposition}

\begin{proof}
    (1) By the proposition below, there exists some constants $c_k$ for $0\le k\le n$, such that 
    $$L_{pq}=\sum\limits_{k=0}^nc_k(p+\frac{1}{2})^k(q+\frac{1}{2})^k,$$
    which implies the statement.\\
    (2) It is equivalent to show that
    $$\sum\limits_{i=0}^p(-1)^{p-i}\binom{p}{i}(i-\frac{1}{2})^k=0$$
    for any $k<p$, which follows from the fact that the Stirling number of the second kind $S(k,p)=\frac{1}{p!}\sum\limits_{i=1}^p(-1)^{p-i}\binom{p}{i}i^k=0$ if $p>k$.
\end{proof}

\begin{proposition}
    $$\sum\limits_{a=0}^n\binom{a+(p+1)(q+1)-1}{a}\binom{n-a+pq-1}{n-a}=\frac{1}{n!}\prod_{i=1}^n(2pq+p+q+i).$$
\end{proposition}

We need the following lemmas to prove the above proposition.
\begin{lemma}
    For any $n,k\in \mathbb{Z}_{>0}$, 
    $$\sum\limits_{i=0}^{k}(-1)^i\binom{n}{i}\binom{n+k-1-i}{k-i}=0.$$
\end{lemma}

\begin{proof}
    By induction on $n$ and $k$. More precisely, we can deduce the case of $(n,k)$ by the cases of $(n-1,k-1)$, $(n,k-1)$ and $(n-1,k)$.
\end{proof}

\begin{lemma}\label{lemma 3.3.4}
    For any $n,k\in \mathbb{Z}_{>0}$,$n\ge k$, $x\in \mathbb{C}$,
    $$\sum\limits_{i=0}^{k}(-1)^i i! \binom{n}{i}\binom{k}{i}(\prod_{j=0}^{n-1-i}(x+k+j))=\prod_{j=0}^{n-1}(x+j).$$
\end{lemma}

\begin{proof}
    It is equivalent to prove that $x=0,-1,\dots, -k+1$ are the roots of the left-hand side of the equation. This follows from the previous lemma.
\end{proof}
\

\begin{proof}[Proof of Proposition 3.3.2.]
Let $u=(p+\frac{1}{2})(q+\frac{1}{2})+\frac{1}{4}$, $v=\frac{1}{2}(p+q)$, then
\begin{align*}
    &\sum\limits_{a=0}^n\binom{a+(p+1)(q+1)-1}{a}\binom{n-a+pq-1}{n-a}\\ =&\frac{1}{n!}\sum\limits_{a=0}^n\binom{n}{a}(\prod_{k=1}^a(u+v+k-1))(\prod_{l=1}^{n-a}(u-v+l-1)).
\end{align*}
Define the right-hand side of the above equation as $g(u,v)$. Suppose
$$g(u,v)=\sum\limits_{i=0}^ng_i(u)v^i.$$
We have the equation
$$g(u,u)=\frac{1}{n!}\prod_{k=1}^n(2u+k-1).$$
For any $t\in\mathbb{Z}$ and $0\le t\le n-1$, using Lemma \ref{lemma 3.3.4} we get
\begin{align*}
        g(u,u+t)=\frac{1}{n!}&\sum\limits_{n-a=0}^{t+1}\binom{n}{a}(\prod_{k=1}^a(2u+k-1+t))(\prod_{l=1}^{n-a}(l-1-t))\\
        &=\frac{1}{n!}\prod_{k=1}^n(2u+k-1).
    \end{align*}
\begin{align*}
        g(u,-u-t)=\frac{1}{n!}&\sum\limits_{a=0}^{t+1}\binom{n}{a}(\prod_{k=1}^a(k-1-t))(\prod_{l=1}^{n-a}(2u+l-1+t))\\
        &=\frac{1}{n!}\prod_{k=1}^n(2u+k-1).
    \end{align*}
Fix $u\in \mathbb{C}$, define $h_u(v)=g(u,v)$, the degree of $h$ is not greater than $n$. However, we have
$$h_u(u)=h_u(u+i)=h_u(-u-i)$$
for any $i\in\mathbb{Z}$, $0\le i\le n-1$. Therefore, for $u\ne -n+1,-n+2\dots,n-1$, $h_u$ is a constant.
$$g(u,v)=h_u(v)=h_u(u)=\frac{1}{n!}\prod_{k=1}^n(2u+k-1).$$
Since $g$ is a continuous function, we get
$$g(u,v)=\frac{1}{n!}\prod_{k=1}^n(2u+k-1).$$
for all $u,v\in \mathbb{C}$.
Using $u=(p+\frac{1}{2})(q+\frac{1}{2})+\frac{1}{4}$, $v=\frac{1}{2}(p+q)$, we get the desired formula in the proposition.
\end{proof}

\begin{proof}[Proof of Theorem\ref{main theorem}]
    We have 
    \begin{equation*}
        L=V\cdot D\cdot{}^tV = P\cdot M\cdot{}^tP,
    \end{equation*}
    thus
    \begin{equation*}
        M=Q\cdot D\cdot{}^tQ,
    \end{equation*}
    which gives the Gauss decomposition of $M$. Note that the Vandermonde matrix $V$ is totally positive and 
    \begin{equation*}
        V=P\cdot Q 
    \end{equation*}
    gives the (opposite) Gauss decomposition of $V$, we conclude by Whitney's theorem\cite{c}.
\end{proof}

\begin{remark}
    The above Gauss decomposition of $M$ is similar to \cite[Proposition A.6.]{ks}. In fact, for types $A$ and $B$, the Gauss decomposition of the metamatrix $M(W)$ has a uniform description 
    \begin{equation*}
        M(W)=\frac{1}{|W|}{}^tQ(W)\cdot D(W)\cdot Q(W).
    \end{equation*}
    We learned the following combinatorial interpretation from Tao Gui.
    The entries in $M(W)$ count the two-sided face numbers of T. K. Peterson's two-sided analog of the Coxeter complex \cite{P}, each column of the upper triangular matrix $Q(W)$ gives the f-vector of the corresponding $W$-permutohedron of lower ranks of the same type with $W$, the entries of the diagonal matrix $D(W)$ compute the betti numbers of the complement of the complexified Coxeter arrangement of type $W$. It is interesting to know whether there is a conceptional explanation and a type-uniform proof of the above decomposition and whether it can be generalized to other types.
\end{remark}

\section{Exceptional types}

\subsection{Strategy for exceptional types}

The following proposition is a classical result.
\begin{proposition}
    Let $(W,S)$ be a finite Coxeter group, $I,J$ be two subsets of $S$. Let $l:W\rightarrow \ZZ$ be the length function on $W$. There is a canonical bijection
    \begin{equation*}
        W_I\backslash W/ W_J \longrightarrow {}^I W ^J,
    \end{equation*} 
    where     
    \begin{equation*}
        {}^I W ^J := \{w\in W: I\subset L(w),J\subset R(w)\},
    \end{equation*}
    where 
    \begin{align*}
        L(w):=\{i\in S: l(s_iw)>l(w)\},\\
        R(w):=\{i\in S: l(ws_i)>l(w)\}.
    \end{align*}
\end{proposition}

By this proposition, we have the following corollary, 
\begin{corollary}
    \begin{equation*}
        M_{pq} = \sum_{|I|=p,|J|=q} |{}^I W^J| = \sum_{i,j} \binom{i}{p}\binom{j}{q}N_{ij},
    \end{equation*}    
    where 
    \begin{equation*}
        N_{ij} := \#\{ w\in W: \#L(w)=i,\#R(w)=j\}.
    \end{equation*}
\end{corollary}

It is not difficult for computer to go through elements of $W$, give the matrix $N(W)=(N_{ij})$ and calculate $M(W)$ by $N(W)$.

\begin{remark}
    The entries $N_{ij}$ in the matrix $N(W)$ are the two-sided $W$-Eulerian numbers defined by T. K. Peterson in \cite{P}. There are two kinds of interesting symmetries in these numbers (see \cite[Observation 15]{P}): 
    \begin{equation} \label{sym}
        N_{ij}=N_{ji} \text{ and } N_{ij}=N_{n-i,n-j},
    \end{equation}
    where $n$ is the rank of the Coxeter system $(W,S)$. There is an interesting conjecture by Tao Gui (unpublished) saying that there should be a smooth projective variety $X(W)$ whose Hodge numbers are given by those $N_{ij}$'s. That is, after a rotation of 45 degrees, the matrix $N(W)$ gives the Hodge diamond of $X(W)$. If this conjecture is true, then the first kind of symmetry in \eqref{sym} comes from the Hodge symmetry whereas the second kind of symmetry in \eqref{sym} comes from the Serre duality. Note that T. K. Peterson's generalized Gessel's conjecture \cite[Conjecture 16]{P} asserts that the generating polynomials of the $W$-Eulerian numbers $N_{ij}$ are $\gamma$-positive, which easily implies that each diagonal of the matrix $N(W)$ is unimodal, a numerical shadow of the hard Lefschetz theorem. As far as we know, Gui's conjecture is currently unsolved even for type $A$.
\end{remark}


\subsection{Results}
Contingency metamatrices for exceptional types are all totally positive and they are listed below.
\begin{center}
 Contingency Metamatrix 
\end{center}

$I_2(m),m\ge 2$  
\begin{equation*}
    \begin{psmallmatrix}
    2m & 2m & 1\\
    2m & 2m+2 & 2\\
    1 & 2 & 1
    \end{psmallmatrix}
\end{equation*}

$H_3$ 
\begin{equation*}
    \begin{psmallmatrix}
    120 & 180 & 62 & 1\\
    180 & 288 & 111 & 3\\
    62 & 111 & 52 & 3\\
    1 & 3 & 3 & 1
    \end{psmallmatrix}
\end{equation*}

$H_4$ 
\begin{equation*}
    \begin{psmallmatrix}
    14400 & 28800 & 17040 & 2640 & 1\\
    28800 & 58560 & 35520 & 5764 & 4\\
    17040 & 35520 & 22366 & 3892 & 6\\
    2640 & 5764 & 3892 & 772 & 4\\
    1 & 4 & 6 & 4 & 1
    \end{psmallmatrix}
\end{equation*}

$F_4$ 
\begin{equation*}
    \begin{psmallmatrix}
    1152 & 2304 & 1392 & 240 & 1\\
    2304 & 4800 & 3072 & 580 & 4\\
    1392 & 3072 & 2134 & 460 & 6\\
    240 & 580 & 460 & 124 & 4\\
    1 & 4 & 6 & 4 & 1
    \end{psmallmatrix}
\end{equation*}

$E_6$ 
\begin{equation*}
    \begin{psmallmatrix}
    51840 & 155520 & 172800 & 86400 & 18558 & 1278 & 1\\
    155520 & 497520 & 550800 & 287100 & 65124 & 4830 & 6\\
    172800 & 550800 & 658800 & 361350 & 87680 & 7145 & 15\\ 
    86400 & 287100 & 361350 & 211450 & 55945 & 5165 & 20\\
    18558 & 65124 & 87680 & 55945 & 16650 & 1834 & 15\\
    1278 & 4830 & 7145 & 5165 & 1834 & 268 & 6\\
    1 & 6 & 15 & 20 & 15 & 6 & 1
    \end{psmallmatrix}    
\end{equation*}

$E_7$
\begin{equation*}
    \begin{psmallmatrix}
        2903040 & 10160640 & 13789440 & 9072000 & 2938320 & 415800 & 17642 & 1\\
        10160640 & 36126720 & 49956480 & 33626880 & 11211480 & 1648920 & 73927 & 7\\
        13789440 & 49956480 & 70640640 & 48867840 & 16868580 & 2598930 & 124611 & 21\\
        9072000 & 33626880 & 48867840 & 34960080 & 12595710 & 2055820 & 107265 & 35\\
        2938320 & 11211480 & 16868580 & 12595710 & 4794276 & 843134 & 49183 & 35\\
        415800 & 1648920 & 2598930 & 2055820 & 843134 & 164334 & 11231 & 21\\
        17642 & 73927 & 124611 & 107265 & 49183 & 11231 & 994 & 7\\
        1 & 7 & 21 & 35 & 35 & 21 & 7 & 1\\
    \end{psmallmatrix}
\end{equation*}

$E_8$
\begin{equation*}
    \begin{psmallmatrix}
        696729600 & 2786918400 & 4470681600 & 3657830400 & 1601268480  & 357557760 & 34508640 & 881760 & 1\\
        2786918400 & 11240570880 & 18207866880 &  15071616000 & 6692958720 & 1522152576 & 150602304 & 4006856 & 8\\
        4470681600 & 18207866880 & 29831639040 & 25032430080 & 11304830880 & 2627041536 & 267654842 & 7467894 & 28\\
        3657830400 & 15071616000 & 25032430080 & 21351747648 & 9839303040 & 2346581468 & 247700376 & 7318836 & 56\\
        1601268480 & 6692958720 & 11304830880 & 9839303040 & 4648819998 & 1144964066 & 126314765 & 4008367 & 70\\
        357557760 & 1522152576 & 2627041536 & 2346581468 & 1144964066 & 293984848 & 34351972 & 1196498 & 56\\
        34508640 & 150602304 & 267654842 & 247700376 & 126314765 & 34351972 & 4349062 & 172685 & 28\\
        881760 & 4006856 & 7467894 & 7318836 & 4008367 & 1196498 & 172685 & 8524 & 8\\
        1 & 8 & 28 & 56 & 70 & 56 & 28 & 8 & 1\\
    \end{psmallmatrix}
\end{equation*}

\

\subsection*{Acknowledgments}
This is a project work during the 2024 ``Algebra and Number
Theory” summer school. We extend our appreciation to the organizers Academy of Mathematics and Systems Science, Chinese Academy of Sciences, and Peking University. We also
thank Michael Finkelberg for bringing us to the conjecture of Mikhail Kapranov and Vadim Schechtman. Furthermore, we are grateful to our mentor Ruotao Yang for the comprehensive assistance rendered
during the completion of this project. Lastly, we would like to express our thanks to Tao Gui 
for the inspiring combinatorial background and comments for the writing up, and Heng Yang for the help in programming.

\end{document}